\documentclass[12pt]{article}
\usepackage[utf8]{inputenc}
\usepackage{amsmath}
\usepackage{amsthm}
\usepackage{amscd}
\usepackage{comment}
\usepackage[english]{babel}
\usepackage{color,graphicx,amssymb,tikz-cd,wrapfig}
\usepackage{amsmath,amsthm,amsfonts,amssymb,stmaryrd}
\usepackage{bm}
\usepackage[rightcaption]{sidecap}

\usepackage{enumerate}
\usepackage{pb-diagram}
\usepackage{graphicx} 
\usepackage[all]{xy}
\usepackage{appendix}

\title{Quantum Gromov-Hausdorff convergence of spectral truncations for groups with polynomial growth}

\author{Ryo Toyota \thanks{Department of Mathematics, Texas A\&M University, College Station, TX 77843, USA}}
\date{ }

\begin{document}

\newtheorem{dfn}{Definition}[section]
\newtheorem{thm}{Theorem}[section]
\newtheorem{lem}{Lemma}[section]
\newtheorem{cor}{Corollary}[section]
\newtheorem{rmk}{Remark}[section]
\newtheorem*{pf}{Proof}
\newtheorem{ex}{Example}[section]
\newtheorem{prop}{Proposition}[section]
\newtheorem*{claim}{Claim}

\newcommand{\tr}{\text{tr}}

\maketitle

\begin{abstract}
For a unital spectral triple $(\mathcal{A}, H,D)$, we study when its truncation converges to itself. The spectral truncation is obtained by using the spectral projection $P_{\Lambda}$ of $D$ onto $[-\Lambda,\Lambda]$ to deal with the case where only a finite range of energy levels of a physical system is available. By restricting operators in $\mathcal{A}$ and $D$ to $P_{\Lambda}H$, we obtain a sequence of operator system spectral triples $\{(P_{\Lambda}\mathcal{A}P_{\Lambda},P_{\Lambda}H,P_{\Lambda}DP_{\Lambda})\}_{\Lambda}$. We prove that if the spectral triple is the one constructed using a discrete group with polynomial growth, then the sequence of operator systems $\{P_{\Lambda}\mathcal{A}P_{\Lambda}\}_{\Lambda}$ converges to $\mathcal{A}$ in the sense of quantum Gromov-Hausdorff convergence with respect to the Lip-norm coming from high order derivatives.
\end{abstract}

\tableofcontents

\section{Introduction}
In this paper, we study quantum Gromov-Hausdorff convergences of operator systems given by spectral truncations. The study of spectral truncation is initiated by A. Connes and W. D. van Suijlekom in \cite{connes2021spectral} to approximate the original spectral triple from the partial data of spectrum of Dirac operator (or Hamiltonian). More precisely, assume that we are given a unital spectral triple $(\mathcal{A},H,D)$ and for each positive number $\Lambda$, let $P_{\Lambda}\in B(H)$ be the projection onto the space spanned by eigenvectors of $D$ corresponding to eigenvalues in $[-\Lambda,\Lambda]$.
Spectral truncation is the procedure to restrict operators in $\mathcal{A}$ and $D$ to $P_{\Lambda}H$ to obtain an operator system version of a spectral triple $(P_{\Lambda}\mathcal{A}P_{\Lambda},P_{\Lambda}H,P_{\Lambda}DP_{\Lambda})$ of \cite{connes2021spectral}. The main purpose of this article is to study when the truncated operator systems $ P_{\Lambda}\mathcal{A}P_{\Lambda}$ approximate the original algebra $\mathcal{A}$ in the language of compact quantum metric space defined by Marc A. Rieffel in \cite{rieffel1999metrics} and \cite{rieffel2004gromov}.

For the spectral triple $(C^{\infty}(M),L^2(S_M),D_M)$ of a compact spin Riemannian manifold $M$ with spinor bundle $S_M$ and the Dirac operator $D_M$, we can recover the distance on $M$ from the spectral data by the Connes distance formula (Formula 1 of \cite{connes1994noncommutative} of section 6.1)
\begin{equation}
\label{diatance}
    d(x,y)=\sup\{|f(x)-f(y)|:f \in C^{\infty}(M),\|[D_M,f]\|\leq 1\},
\end{equation}
for $x,y \in M$.
If we regard the evaluation at $x \in M$ to a function $f \in C^{\infty}(M)$ as the evaluation of a function $f$ to the delta function $\delta_x$ supported at $x$, we can extend the formula \eqref{diatance} to the distance of state spaces. Namely for two states $\phi$ and $\psi$ on $C^{\infty}(M)$, we define the distance between them by
\begin{equation}
\label{stateconnesdistance}
    d(\phi,\psi)=\sup\{|\phi(f)-\psi(f)|:f \in C^{\infty}(M),\|[D_M,f]\|\leq 1\}.
\end{equation}
The idea of compact quantum metric space is to replace a compact metric space $(X,d)$  by a pair $(A,L)$ consists of an operator system $A$ (or Archimedean ordered unit vector space more generally) and a seminorm $L$ on $A$ so called a Lip-norm to define a distance on the state space of $A$ by the formula
\begin{equation}
\label{lipdistance}
    d(\phi,\psi)=\sup\{|\phi(f)-\psi(f)|:f \in A,L(f)\leq 1\},
\end{equation}
for two states $\phi$ and $\psi$ on $A$. We call such a pair $(A,L)$ a compact quantum metric space. Philosophically, Lip-norm $L$ is a norm of commutator with derivative (or Dirac operator) but we are more flexible and the actual requirements for a seminorm on $A$ to be a Lip-norm is in Definition 2.1. The most important one is that the distance defined by \eqref{lipdistance} has to induce weak-$*$ topology on the state space. Since we are working on operator systems, we can not impose Leibniz inequality on $L$ (for example, the norm of high order commutator with Dirac operator can be a Lip-norm in some cases \cite{antonescu2004metrics}.) For given two compact quantum metric spaces $(A,L_A)$ and $(B,L_B)$, we define the quantum Gromov-Hausdorff distance between them by a certain infimum of Hausdorff distance between their state spaces. We apply this framework to the operator systems obtained by spectral truncations. In \cite{leimbach2023gromov}, Gromov-Hausdorff convergence of the state speces of truncated operator systems $P_{\Lambda}\mathcal{A}P_{\Lambda}$ to the state spaces of the algebra $\mathcal{A}$ was proved when the spectral triple is the torus $(\mathcal{A},H,D)=(C^{\infty}(\mathbb{T}^d),L^2(S_{\mathbb{T}^d}),D_{\mathbb{T}^d})$ for $d=1,2,3$ using the metric defined in \eqref{stateconnesdistance}.

In this paper, we deal with spectral triples $(\mathbb{C}[G],\ell^2(G),D)$ consisting of the group algebra $\mathbb{C}[G]$ of a finitely generated group $G$ with a length function $\ell$, Hilbert space $\ell^2(G)$ and the Dirac operator $D$ that is the multiplication of the length function:
\begin{align*}
    D:\ell^2(G)\rightarrow \ell^2(G): \text{ }\delta_x \mapsto \ell(x)\delta_x. 
\end{align*}
Note that in this case, for each positive $\Lambda$, the spectral projection $P_{\Lambda}$ is the orthogonal projection onto the space spanned by vectors supported in the closed ball $B_{\Lambda}$ centered at the unit $e$ with radius $\Lambda$. The truncated operator system consists of operators of the form of Toeplitz type matrices:
\begin{equation}
\label{toeplitz}
    P_{\Lambda}\mathbb{C}[G]P_{\Lambda}:=\{(a_{xy^{-1}})_{x,y\in B_{\Lambda}}:\text{ }a_z\in \mathbb{C} \text{ for all }z\in B_{2\Lambda}\}.
\end{equation}

We use Lip-norms coming from high order derivatives. The use of high order derivatives in compact quantum metric spaces was initiated in \cite{antonescu2004metrics} and it produces new examples of compact quantum metric spaces. But our Lip-norm is slightly different from the one introduced in \cite{antonescu2004metrics} (the reason is in Remark 3.1). For each positive integer $s$, we define Lip-norms on group algebras and truncated operator systems by
\begin{align*}
    L_s(f)&:=\|\sum \ell(x)^s f(x)\delta_x\|_{B(\ell^2(G))}\\
    L_{s,\Lambda}(A)&:=\|(\ell(xy^{-1})^sa_{xy^{-1}})_{x,y\in B_{\Lambda}}\|_{B(P_{\Lambda}\ell^2(G))}
\end{align*}
for $f \in \mathbb{C}[G]$ and $A=(a_{xy^{-1}})_{x,y \in B_{\Lambda}}\in P_{\Lambda}\mathbb{C}[G]P_{\Lambda} $. In Lemma 3.1, we verify that these seminorms satisfy the requirement of Definition 2.1 for being Lip-norm. Now we can state our main theorem.
\begin{thm}
\rm{
    For each group $G$ with polynomial growth, there exists $s>0$ such that the sequence of compact quantum metric spaces $\{(P_{\Lambda}\mathbb{C}[G]P_{\Lambda},L_{s,\Lambda})\}_{\Lambda}$ converges to $(\mathbb{C}[G],L_s)$ in quantum Gromov-Hausdorff distance.
    }
\end{thm}

\section{Preliminaries about compact quantum metric spaces}

In this section, we recall the notion of compact quantum metric spaces and quantum analogue of Gromov-Hausdorff convergence introduced by Marc A Rieffel in \cite{rieffel1999metrics} and \cite{rieffel2004gromov}. We recall some definitions and properties introduced in \cite{rieffel2023convergence}.

\begin{dfn}
\rm{
    Let $(A,e)$ be a pair of an operator system $A$ with its unit $e$. Then a seminorm $L$, which may take value $\infty$ on $A$ is called a Lip-norm if $L$ satisfies the following conditions:
    \begin{itemize}
        \item For $a \in A$, we have $L(a)=L(a^*)$, and $L(a)=0$ if and only if $a \in \mathbb{C}e$.
        \item $\{a \in A;\text{ }L(a)< \infty\}$ is dence in $A$.
        \item Define a distance $d^L$ on the state space $S(A)$ by
            \begin{align*}
                d^L(\phi,\psi):=\sup\{|\phi(a)-\psi(a)|:\text{ }a\in A \text{ and }L(a)\leq 1  \}
            \end{align*}
            for any states $\phi$ and $\psi$ on $A$. We require that the induced topology on $S(A)$ by the distance $d^L$ agrees with the weak-$*$ topology.
        \item $L$ is lower semi-continuous with respect to the operator norm.
    \end{itemize}

    We call a pair $(A,L)$ of operator system $A$ with a Lip-norm $L$ a compact quantum metric space.
    }
\end{dfn}

In the next definition, we recall the definition of quantum Gromov-Hausdorff distance between two compact quantum metric spaces $(A,L^A)$ and $(B,L^B)$. We denote the coordinate quotient maps by
 \begin{align*}
     q_A:A\oplus B \rightarrow A\text{ and } q_B:A \oplus B \rightarrow B.
 \end{align*}
Each Lip-norm $L$ on $A\oplus B$ induces a seminorm $q_A^*(L)$ (resp. $q_B^*(L)$) on $A$ (resp. on $B$) by
\begin{align*}
    q_A^*(L)(a):=\inf \{L(a,b):b\in B\} \text{ (resp. } q_B^*(L)(b):=\inf \{L(a,b):a\in A\}\text{)}.
\end{align*}
There are embeddings of state spaces
\begin{align*}
    q_A^*:S(A)\hookrightarrow S(A\oplus B) \text{ and } q_B^*:S(B) \hookrightarrow S(A \oplus B)
\end{align*}
defined by the pullbacks by $q_A$ and $q_B$. For a Lip-norm $L$ on $A\oplus B$, we denote the Hausdorff distance between $q_A^*S(A)$ and $q_B^*S(B)$ in the compact metric space $(S(A\oplus B),d^L)$ by 
\begin{align*}
    \text{dist}_H^{d^L}(q_A^*S(A),q_B^*S(B)).
\end{align*}

\begin{dfn}
\rm{
    Let $\mathcal{M}(L_A,L_B)$ be the set of all Lip-norm on $A\oplus B$ which induces $L_A$ and $L_B$. Then the quantum Gromov-Hausdorff distance between $(A,L_A)$ and $(B,L_B)$ is defined to be
    \begin{align*}
        \text{dist}_q((A,L_A),(B,L_B)):=\inf \{ \text{dist}_H^{d^L}(S(A),S(B)): \text{ }L\in \mathcal{M}(L_A,L_B) \}.
    \end{align*}
}
\end{dfn}

We finish this section by stating a criterion of quantum Gromov-Hausdorff convergence, which is used in \cite{rieffel2004gromov} to show the quantum Gromov-Hausdorff convergence of matrix algebras to $2$-sphere and explicitly proved in \cite{van2021gromov} as a sufficient condition for Gromov-Hausdorff convergences of state spaces.

\begin{lem}
\rm{
    Let $(A,L_A)$ and $(B,L_B)$ be compact quantum metric spaces. Assume that there are unital positive maps which are contractive with respect to their Lip-norms
    \begin{align*}
        r:A \rightarrow B \text{ and } s:B \rightarrow A
    \end{align*}
    such that
    \begin{align*}
        \|s \circ r(a)-a\| &\leq \epsilon L_A(a) \\
        \|r \circ s(b)-b\| &\leq \epsilon L_{B}(b).
    \end{align*}
    Then the quantum Gromov-Hausddorff distance between $(A,L_A)$ and $(B,L_B)$ is less than $2\epsilon$.
    }
\end{lem}

The above criterion should be well-known among experts but the author was not able to find any precise reference. For the convenience of readers, we give a proof of the above lemma. 

We use the next concept and result introduced in \cite{rieffel2004gromov} (the second half of the book, which consists of two papers).

\begin{dfn}[Definition 5.1 of \cite{rieffel2004gromov}]
    \rm{
    Let $(A,L_A)$ and $(B,L_B)$ be compact quantum metric spaces. By a bridge between $(A,L_A)$ and $(B,L_B)$ we mean a seminorm, $N$, on $A\oplus B$ such that
    \begin{description}
        \item [1)] $N$ is continuous for the norm on $A\oplus B$.

        \item [2)] $N(e_A,e_B)=0$ but $N(e_A,0)\neq 0$.

        \item [3)] For any $a\in A$ and $\delta>0$ there is $b \in B$ such that
        \begin{align*}
            L_B(b) \lor N(a,b) \leq L_A(a) +\delta,
        \end{align*}
        and similarly for $A$ and $B$ interchanged. Here the symbol $\lor$ means "maximum".
    \end{description}
    }
\end{dfn}

\begin{thm}[Theorem 5.2 of \cite{rieffel2004gromov}]
\rm{
    Let $N$ be a bridge between a compact quantum metric spaces $(A,L_A)$ and $(B,L_B)$. Define $L$ on $A\oplus B$ by
    \begin{align*}
        L(a,b)=L_A(a)\lor L_B(b) \lor N(a,b).
    \end{align*}
    Then $L$ is a Lip-norm which induces $L_A$ and $L_B$.
    }
\end{thm}

\begin{lem}
\rm{
    Under the assumption of Lemma 2.1, we define a seminorm on $A \oplus B$ by 
    \begin{align*}
        N(a,b):=\frac{1}{\epsilon}\|a-s(b)\|.
    \end{align*}
    Then $N$ is a bridge between $(A,L_A)$ and $(B,L_B)$.
    }
\end{lem}
\begin{pf}
\rm{
    We only need to prove the third condition in Definition 2.3. We choose $b=r(a)$, then
    \begin{align*}
        L_B(b) \lor N(a,b)&=L_B(r(a)) \lor \frac{1}{\epsilon}\|a-s\circ r(a)\| \\
        &= L_A(a) \lor \frac{1}{\epsilon}\cdot (\epsilon L_A(a))=L_A(a).
    \end{align*}
    }
    We verify the same condition for $A$ and $B$ interchanged. Fix any $b\in B$. We choose $a=s(b)$. Then we have
    \begin{align*}
        L_A(a) \lor N(a,b) = L_A(s(b)) \lor \frac{1}{\epsilon} \|s(b)-s(b)\|\leq L_B(b).
    \end{align*}
    \qed
\end{pf}

\begin{proof}[Proof of Lemma 2.1]
\rm{
Let $L$ be the Lip-norm defined in Theorem 2.1 for the bridge defined in Lemma 2.2. It suffices to show the Hausdorff distance $ \text{dist}_H^{d^L}(q_A^*S(A),q_B^*S(B)) $ is smaller than $2\epsilon$.
First take any $\mu \in S(A)$. We show
\begin{align*}
    d^L(q_A^*(\mu), q_B^*s^*(\mu)) \leq \epsilon.
\end{align*}
For any $(a,b)\in A\oplus B$ with $L(a,b)\leq 1$, we have
\begin{align*}
    |q_A^*(\mu)(a,b)-q_B^*s^*(\mu)(a,b)|=|\mu(a)-\mu(s(b))|\leq \|\mu\|\cdot \|a-s(b)|\leq \epsilon,
\end{align*}
because of the definition of $N$. This means $q_A^*S(A)$ is contained in the $\epsilon$-neighborhood of $q_B^*S(B)$.
For the other direction, we take any $\nu\in S(B)$. We show 
\begin{align*}
    d^L(q_A^*r^*(\nu), q_B^*(\nu)) \leq 2\epsilon.
\end{align*}
For any $(a,b)\in A\oplus B$ with $L(a,b)\leq 1$, we have
\begin{align*}
    |q_A^*r^*(\nu)(a,b)- q_B^*(\nu)(a,b)|&=|\nu(r(a))-\nu(b)|\\
    &\leq  |\nu(r(a))-\nu(r\circ s(b))|+|\nu(r\circ s(b))-\nu(b)|\\
    &\leq \|\nu\|\cdot \|a-s(b)\|+ \|\nu\|\cdot \|r\circ s(b)-b\| \leq 2\epsilon.
\end{align*}
This means $q_B^*S(B)$ is contained in the $2\epsilon$-neighborhood of $q_A^*S(A)$. 
}
\end{proof}

\section{Group $C^*$-algebras as compact quantum metric spaces}

In this section, we introduce quantum metric structures on group $C^*$-algebras and their spectral truncations. For a discrete group $G$, there is a natural Dirac operator on $\ell^2(G)$ introduced by A. Connes in \cite{connes1989compact}. The seminorm defined to be the norm of commutator with the Dirac operator has been studied in \cite{ozawa2005hyperbolic},\cite{antonescu2004metrics} and \cite{christ2017nilpotent}. But as shown in Remark 3.1, this definition is not a Lip-norm on truncated operator systems in the sense of Definition 2.1. Therefore we introduce a different definition of Lip-norm on group $C^*$-algebras. 

We first introduce several notations for discrete groups which will be used later. Let $G$ be a group generated by a finite symmetric subset $S=S^{-1} \subset G$. Then this generator $S$ induces a word length $\ell:G\rightarrow\mathbb{R}_{+}$ and a distance $d$ on $G$ defined by $d(x,y)=\ell(x^{-1}y)$ for all $x,y \in G$. The closed ball of radius $\Lambda$ centered at $x \in G$ is denoted by $B_{\Lambda}(x)$ and especially when $x=e$ it is denoted by $B_{\Lambda}$. We denote the left regular representation of group algebra $\mathbb{C}[G]$ on $\ell^2(G)$ by $\lambda$. 

The following example provides an interesting class of compact quantum metric spaces.
\begin{ex}
\rm{
    Let $G$ be a finitely generated group with a word length $\ell$. We define a Dirac operator on $\ell^2(G)$ by
    \begin{align*}
        D:\ell^2(G) \rightarrow \ell^2(G); \text{ } \sum f(x)\delta_x \rightarrow \sum \ell(x)f(x)\delta_x.
    \end{align*}
    Then we define a seminorm on a group algebra $\mathbb{C}[G]$ by
    \begin{align*}
        L'(f):=\|[D,\lambda(f)]\| <\infty
    \end{align*}
    for $f \in \mathbb{C}[G]$. It is shown that if $G$ is hyperbolic (\cite{ozawa2005hyperbolic}) or polynomial growth (\cite{christ2017nilpotent}), then $L'$ is a Lip-norm.
    }
\end{ex}

For a spectral triple $(\mathcal{A},H,D)$, the natural choice of Lip-norm is $L'(a):=\|[D,a]\|$ (if it satisfies the conditions in Definition 2.1) because in commutative case this distance coincides with the distance of the base space by the Connes distance formula \eqref{diatance}. But we will not use this Lip-norm in this paper. The reason is that this does not give a Lip-norm for truncated operator system as we see in the next remark.

\begin{rmk}
\rm{
    Let $(\mathcal{A},H,D)=(\mathbb{C}[G],\ell^2(G),D)$ be the spectral triple of a discrete group $G$, where the Dirac operator $D$ is the one defined in Example 3.1. Then we can form an operator system spectral triple $(P_{\Lambda}\mathcal{A}P_{\Lambda},P_{\Lambda}H,P_{\Lambda}DP_{\Lambda})$. In this case, the operator system $P_{\Lambda}\mathcal{A}P_{\Lambda}$ is specified in \eqref{toeplitz}.
    We define a seminorm $L'_{\Lambda}$ on $P_{\Lambda}\mathcal{A}P_{\Lambda}$ by
    \begin{align*}
        L'_{\Lambda}(A):=\|[P_{\Lambda}DP_{\Lambda},A]\|
    \end{align*}
    for $A=(a_{xy^{-1}})_{x,y \in B_{\Lambda}}$. Let's say $G=\mathbb{Z}$ and define a matrix $A\in P_{\Lambda}\mathbb{C}[G]P_{\Lambda}$ by $a_{2\Lambda}=1$ and $a_n=0$ for $n=-2\Lambda,-2\Lambda+1, \cdots, 2\Lambda-1$. Then 
    \begin{align*}
        [P_{\Lambda}DP_{\Lambda},A]=((\ell(x)-\ell(y))\cdot a_{x-y})_{x,y=-\Lambda}^{\Lambda}.
    \end{align*}
    By construction, the $(i,j)$-entry of above matrix is $0$ if $(i,j)\neq (\Lambda,-\Lambda)$ but for $(\Lambda,-\Lambda)$-entry, we have $\ell(\Lambda)-\ell(-\Lambda)=0$. So $ [D_{\Lambda},A]=0 $ and $L'_{\Lambda}$ has a nontrivial kernel other than scalar matrices.
    }
\end{rmk}

We introduce another seminorms, which might be more natural analogue of Lipschitz norm of commutative case. We fix the notations in the next definition through this paper.
\begin{dfn}
\rm{
    We define derivatives on group algebra $\mathbb{C}[G]$ and its truncated operator systems $P_{\Lambda}\mathbb{C}[G]P_{\Lambda}$ by
    \begin{align*}
        d : \mathbb{C}[G] \rightarrow \mathbb{C}[G] &;\text{ }\sum f(x)\delta_x \rightarrow \sum \ell(x)f(x)\delta_x \\
        d_{\Lambda}:P_{\Lambda}\mathbb{C}[G]P_{\Lambda} \rightarrow P_{\Lambda}\mathbb{C}[G]P_{\Lambda}&;\text{ } (a_{xy^{-1}})_{x,y \in B_{\Lambda}} \mapsto (\ell(xy^{-1})a_{xy^{-1}})_{x,y \in B_{\Lambda}}.
    \end{align*}
    For each positive integer $s$, we define seminorms by $L_s(f):=\|d^sf\|_{B(\ell^2(G))}$ and $L_{s,\Lambda}(T):=\|d_{\Lambda}^sT\|_{B(P_{\Lambda}\ell^2(G))}$ for $f\in \mathbb{C}[G]$ and $T\in P_{\Lambda}\mathbb{C}[G]P_{\Lambda} $.
    }
\end{dfn}

\begin{lem}
\rm{ 
    For any finitely generated group $G$ and $s$, $L_s$ is lower semicontinuous with respect to the reduced $C^*$-norm. 
}
\end{lem}

\begin{pf}
    \rm{
    Assume $\{f_j\}_j$ is a sequence in $ \mathbb{C}[G] $ that converges to $f \in \mathbb{C}[G]$ with respect to the operator norm and $L_s(f_j)\leq 1$ for all $j$. For any subset $F \subset G$, let $P_F\in B(\ell^2(G))$ be the projection onto the space $\ell^2(F)\subset \ell^2(G)$. Then for any finite subset $F \subset G$, we have
    \begin{align*}
        \left\|P_F (\lambda(\sum_x \ell(x)^sf(x)\delta_x))P_F \right\|_{B(\ell^2(G))} =\lim_j \left\|P_F \lambda(d^sf_j) P_F\right\|_{B(\ell^2(G))} \leq 1
    \end{align*}
    Since the operator norm of $ \lambda(\sum_x \ell(x)^sf(x)\delta_x) $ is the supremum of the left hand side of the above formula for all finite subset $F \subset G$, we have shown that $L_s(f)\leq 1$. \qed

    }
\end{pf}

We need to show that the seminorm defined above is a Lip-norm. The next Lemma is an analogue of Theorem 2.6 of \cite{antonescu2004metrics} for our seminorms and the same proof works for our case. First we recall the definition of rapid decay property of descrete groups, which appears in the assumption of the next lemma.

\begin{dfn}[\cite{valette2002introduction} Chapter 8]
\rm{
    Let $G$ be a finitely generated group with the length function $\ell$. For $s>0$, we define a weighted $\ell^2$-norm by
    \begin{align*}
        \|f\|_{H^s}:=\left (\sum_{x \in G} |f(x)|^2(1+\ell(x))^{2s} \right )^{\frac{1}{2}}.
    \end{align*}
    for $f \in \mathbb{C}[G]$.
    We say that the group $G$ has rapid decay property if there exist $C>0$ and $s>0$ such that
    \begin{align*}
        \|\lambda(f)\|_{C^*_r(G)} \leq C\|f\|_{H^s}
    \end{align*}
    for all $f \in \mathbb{C}[G]$.
    }
\end{dfn}

\begin{lem}
\rm{
    If $G$ is a group with rapid decay property, then there exists a positive integer $s_0$ such that $L_s$ is a Lip-norm on $\mathbb{C}[G]$ for every $s \geq s_0$.
    }
\end{lem}

\begin{pf}
\rm{
    We show the third condition in Definition 2.1. Since $G$ satisfies the Haagerup inequality, there exist constants $C$ and $s_0$ such that
    \begin{align*}
        \|\lambda(f)\|^2_{B(\ell^2(G))} \leq C \sum_x (1+\ell(x))^{2s_0} |f(x)|^2
    \end{align*}
    By Proposition 1.3 of \cite{ozawa2005hyperbolic}, it suffices to show that the set
    \begin{align*}
        \mathcal{D}_s:=\{f \in \mathbb{C}[G]: \text{ }L_s(f)\leq 1 \text{, } f(e)=0\}
    \end{align*}
    is totally bounded with respect to the operator norm for sufficiently large $s$ depending on $G$. Let $\epsilon>0$ be given.
    We show the following claim
    \begin{claim}
    \rm{
        For any $s \geq s_0+1$, there exists $N$ such that for any $f \in \mathcal{D}_s$ we have
    \begin{align*}
        \left \| \lambda\left (\sum_{\ell(x)\geq N} f(x)\delta_x\right ) \right \|_{B(\ell^2(G))} \leq \epsilon .
    \end{align*}
    }
    \end{claim}
    \begin{proof}[Proof of Claim]
         For $f \in \mathcal{D}_s$, we have
    \begin{align*}
        \left \| \lambda\left (\sum_{\ell(x)\geq N} f(x)\delta_x \right) \right \|^2_{B(\ell^2(G))}&\leq C\sum _{\ell(x)\geq N} (1+\ell(x))^{2s_{0}} |f(x)|^2 \\
        & \leq C\cdot 2^{s_0} \sum _{\ell(x)\geq N} \ell(x)^{2s_0} |f(x)|^2 \\
        & =  C\cdot 2^{s_0} \sum _{\ell(x)\geq N} \ell(x)^{2s_0-2s} \ell(x)^{2s} |f(x)|^2 \\
        & \leq C\cdot 2^{s_0} \cdot \frac{1}{N^{2(s-s_0)}} \sum _{\ell(x)\geq N} \ell(x)^{2s} |f(x)|^2 \\
         & \leq C\cdot 2^{s_0} \cdot \frac{1}{N^{2(s-s_0)}} \sum _{x \in G} \ell(x)^{2s} |f(x)|^2 \\
        & \leq C\cdot 2^s \cdot \frac{1}{N^{2(s-s_0)}} \left \| \lambda\left (d^s f\right ) (\delta_e) \right \|_{\ell^2(G)}\\
        & \leq C\cdot 2^s \frac{1}{N^{2(s-s_0)}} \left \| \lambda\left (d^s f\right )  \right \|_{B(\ell^2(G))}\\
        & \leq \frac{ C\cdot 2^s }{N^{2(s-s_0)}}.
    \end{align*}
    Since $2s-2s_0\geq 0$, by taking $N$ such that $ \frac{ C\cdot 2^s }{N^{2(s-s_0)}}<\epsilon $, the claim follows. 
    \end{proof}
   Since $G$ is finitely generated $\{x\in G: \text{ }\ell(x)<N\}$ is a finite set so
    \begin{align*}
        \{f \in \mathcal{D}_s: \text{ }\text{supp}(f)\subset B_N\}
    \end{align*}
    is totally bounded. Therefore, $\mathcal{D}_s$ is totally bounded.
    \qed
    }
\end{pf}

Since the kernel of $L_{s,\Lambda}$ is only scalar and the truncated operator systems $ P_{\Lambda}\mathbb{C}[G]P_{\Lambda} $ is finite dimensional, we have the following.

\begin{lem}
\rm{
    For every finitely generated group $G$ and for every positive integer $s$, the seminorm $ L_{s,\Lambda} $ is a Lip-norm for $ P_{\Lambda}\mathbb{C}[G]P_{\Lambda} $.
    }
\end{lem}

In the last section, we work on groups with polynomial growth. Here we remark a theorem on an asymptotic estimate of the volume of boundary by R.Tessera.
\begin{thm}[R. Tessera \cite{tessera2007volume}]
\rm{
    Let $G$ be a compactly generated group with polynomial growth with Haar measure $\mu$. Then there exist a constant $C$ and $\beta>0$ such that
    \begin{align*}
        \frac{\mu( B_{\Lambda+1}\setminus B_{\Lambda})}{\mu (B_{\Lambda})} \leq C\Lambda^{-\beta}.
    \end{align*}
    }
\end{thm}

\section{Noncommutative Fejer kernel}

In this section, we develop a theory of noncommutative Fejer kernel for discrete groups to prepare for some estimates we need in the last section.  For a subset $F \subset G$, its cardinality is denoted by $\#F$. For positive number $\Lambda$, we define a finitely supported function $F_{\Lambda}\in \mathbb{C}[G]$ by
\begin{align*}
    F_{\Lambda}(x):=\frac{\#(B_{\Lambda}\cap B_{\Lambda}(x))}{\#B_{\Lambda}}.
\end{align*}
This gives a pointwise multiplication on the group algebra $\mathbb{C}[G]$, which is denoted by 
\begin{align*}
    \mathcal{F}_{\Lambda} : \mathbb{C}[G] \rightarrow \mathbb{C}[G] ; \text{ } \sum f(x)\delta_x \rightarrow  \sum F_{\Lambda}(x)f(x)\delta_x.
\end{align*}

We regard $\mathcal{F}_{\Lambda}$ as a noncommutative analogue of Fejer Kernel introduced for $G=\mathbb{Z}$ to approximate the delta function by Laurent polynomials on $S^1$ with finite degrees. In particular, it has the following approximation property (Lemma 10 of \cite{van2021gromov});\\
"there exists a sequence $\{\epsilon_{\Lambda}\}_{\Lambda}$ which converges to $0$ as $\Lambda$ goes to infinity such that
\begin{align*}
    \|f-\mathcal{F}_{\Lambda}(f)\|_{\infty}\leq \epsilon_{\Lambda}\|f'\|_{\infty}
\end{align*}
for all $f\in C^{\infty}(S^1)\subset C(S^1)\cong C^*_r(\mathbb{Z})$."

But in \cite{leimbach2023gromov}, it is pointed out that this approximation property can not be generalized to high dimensional torus. But if we use a high order derivative, we can show this approximaton holds for all group with polynomial growth.
Here we remark that for finitely generated discrete groups, having polynomial growth is equivalent to being amenable and having rapid decay property (\cite{valette2002introduction} Chapter 8).

Before stating the main estimate of this section, we need the following lemma;
\begin{lem}
\rm{
    Let $G$ be a discrete group generated by a finite symmetric subset $S$. If the sequence $\{B_{\Lambda}\}_{\Lambda}$ is a F{\o}lner sequence then there exists a sequence $\{\epsilon_{\Lambda}\}_{\Lambda}$ of positive numbers which converges to $0$ as $\Lambda$ goes to infinity such that
    \begin{align*}
        \|f - \mathcal{F}_{\Lambda}(f)\|_{\ell^2(G)}\leq \epsilon_{\Lambda} \|Df\|_{\ell^2(G)}.
    \end{align*}
    }
\end{lem}

\begin{pf}
\rm{
    Assume $\{B_{\Lambda}\}_{\Lambda}$ is a F{\o}lner sequence. Then there exists a sequence $\{\epsilon_{\Lambda}\}_{\Lambda}$ of positive numbers which converges to $0$ as $\Lambda$ goes to infinity such that
    \begin{align*}
        \frac{\#(B_{\Lambda}\setminus  B_{\Lambda}(x))}{\#B_{\Lambda}} \leq \epsilon_{\Lambda}
    \end{align*}
    for any $x \in S$.
    Therefore, since for any $x=x_1 x_2 \cdots x_n \in G$, 
    \begin{align*}
        B_{\Lambda}\setminus B_{\Lambda}(x) \subset (B_{\Lambda}\setminus B_{\Lambda}(x_n)) &\amalg (B_{\Lambda}(x_n)\setminus B_{\Lambda}(x_{n-1}x_{n})) \amalg \\\cdots &\amalg (B_{\Lambda}(x_2\cdots x_n)\setminus B_{\Lambda}(x)),
    \end{align*}
    we have
    \begin{equation}
    \label{boundary}
        \frac{\#(B_{\Lambda}\setminus B_{\Lambda}(x))}{\#B_{\Lambda}} \leq \ell(x)  \epsilon_{\Lambda}.
    \end{equation}
    So, for $f \in \mathbb{C}[G]$ we have the following estimate
    \begin{align*}
        &\|f-\mathcal{F}_{\Lambda}(f)\|_{\ell^2(G)}^2\\
        =& \sum_{\ell(x)^2\leq \frac{1}{{\epsilon_{\Lambda}}}}\left( 1- \frac{\#(B_{\Lambda}\cap B_{\Lambda}(x))}{\#B_{\Lambda}} \right)^2|f(x)|^2+ \sum_{\ell(x)^2 > \frac{1}{{\epsilon_{\Lambda}}}}\left( 1- \frac{\#(B_{\Lambda}\cap B_{\Lambda}(x))}{\#B_{\Lambda}} \right)^2|f(x)|^2\\
        \leq& \sum_{0< \ell(x)^2\leq \frac{1}{{\epsilon_{\Lambda}}}}\left( \sqrt{\frac{1}{\epsilon_{\Lambda}}} \cdot \epsilon_{\Lambda}\right)^2|f(x)|^2 +
        \sum_{\ell(x)^2 > \frac{1}{{\epsilon_{\Lambda}}}}|f(x)|^2  \\
        \leq& \epsilon_{\Lambda} \sum_{x \neq e} \ell(x)^2|f(x)|^2.
        \end{align*}
    \qed
    }
\end{pf}

\begin{thm}
\rm{
    If $G$ is a discrete group with polynomial growth, then there exists a positive integer $s$ and a sequence $\{\epsilon_{\Lambda}\}_{\Lambda}$ of positive numbers which converges to $0$ as $\Lambda$ goes to infinity such that 
    \begin{align*}
        \|f-\mathcal{F}_{\Lambda}(f)\|_{B(\ell^2(G))}\leq \epsilon_{\Lambda}  \|d^s(f)\|_{B(\ell^2(G))} 
    \end{align*}
    for any $f \in \mathbb{C}[G]$.
    }
\end{thm}

\begin{pf}
\rm{
    Since $G$ has polynomial growth, the sequence $\{B_{\Lambda}\}_{\Lambda}$ is a F{\o}lner sequence. By applying Lemma 4.1 for $(1+D)^sf$, for any positive integer $s$ (will be fixed later), there exists a sequence $\{\epsilon_{\Lambda}\}_{\Lambda}$ of positive numbers which converges to $0$ as $\Lambda$ goes to infinity such that 
    \begin{equation}
    \begin{split}
        \|(1+D)^{s}(f-\mathcal{F}_{\Lambda}(f))\|^2_{\ell^2(G)} & = \|(1+D)^sf-\mathcal{F}_{\Lambda}((1+D)^sf)\|^2_{\ell^2(G)}\\
        &\leq \epsilon_{\Lambda}^2 \|D(1+D)^{s}f\|^2_{\ell^2(G)}\\
        & = \epsilon_{\Lambda}^2 \sum_{x \neq e} \ell(x)^2(1+\ell(x))^{2s}|f(x)|^2 \\ \label{sobolev}
        &\leq 2^{2s}\epsilon_{\Lambda}^2 \sum_x \ell(x)^{2s+2}|f(x)|^2   \\
        & = 2^{2s}\epsilon_{\Lambda}^2 \|\lambda(d^{s+1}f)\delta_e\|_{\ell^2(G)}^2 \\
        & \leq 2^{2s}\epsilon_{\Lambda}^2 \|\lambda(d^{s+1}f)\|_{B(\ell^2(G))}^2.
    \end{split}
    \end{equation}
    Note that in the first equality, we use the fact $D$ and $\mathcal{F}_{\Lambda}$ commute since both of them are pointwise multiplication.
    Since $G$ satisfies the rapid decay property (Example 8.5 of \cite{valette2002introduction}), the left hand side dominates $\|f-\mathcal{F}_{\Lambda}(f)\|_{C^*_r(G)}$ up to some constant which is independent on $\Lambda$ and $f$ for sufficiently large $s$. This proves the theorem. 
    \qed
}    
\end{pf}

\section{Proof of the main theorem}
 In this section, we prove our main theorem.

 \begin{thm}
 \rm{
     If $G$ is a finitely generated group with polynomial growth, then there exists $s$ such that the truncated operator systems converges to the original algebra in terms of quantum Gromov-Hausdorff distance defined by $s$-Lip-norms; i.e.
     \begin{align*}
     \text{dist}_q((\mathbb{C}[G],L_s),(P_{\Lambda}\mathbb{C}[G]P_{\Lambda},L_{s,\Lambda}))\rightarrow 0.
 \end{align*}
 }
 \end{thm}
 
 Recall from Section 3, the $s$-Lip-norm is a Lip-norm induced by $s$-times derivative. That is,
 \begin{align*}
      L_s(f)&:=\|\lambda(d^sf)\|_{B(\ell^2(G))}\\
      L_{s,\Lambda}(T)&:=\|d_{\Lambda}^sT\|_{B(P_{\Lambda}\ell^2(G))}
 \end{align*}
for $f \in \mathbb{C}[G]$ and for $T\in P_{\Lambda}\mathbb{C}[G]P_{\Lambda}$.

\begin{dfn}
\rm{
We define a unital linear map
    \begin{align*}
        r_{\Lambda}: P_{\Lambda}\mathbb{C}[G]P_{\Lambda} &\rightarrow \mathbb{C}[G] \\
         (f(xy^{-1}))_{x,y \in B_{\Lambda}} &\mapsto \sum_{x \in B_{\Lambda}}f(x)\frac{\#(B_{\Lambda}\cap B_{\Lambda}(x))}{\#B_{\Lambda}} \delta_x.
    \end{align*}
    }
\end{dfn}

For $G =\mathbb{Z}^d$, the positivity of $r_{\Lambda}$ was shown by harmonic analytical techniques in \cite{leimbach2023gromov}. Here we give an operator theoretic proof for general case.

\begin{lem}
\rm{
    The map $r_{\Lambda}$ is unital and completely positive.
    }
\end{lem}

\begin{pf}
\rm{
    We only need to show the completely positivity. 
    For each $\alpha \in G$ we define an operator
    \begin{align*}
        U_{\alpha}:H \rightarrow H
    \end{align*}
    by $(U_{\alpha}\xi)(x)=\xi(x\alpha^{-1})$ for $\xi \in H$ and $x \in G$. Then for any $ A=(f(xy^{-1}))_{x,y \in B_{\Lambda}} \in P_{\Lambda}\mathbb{C}[G]P_{\Lambda} $ and $\xi \in H$, we have
    \begin{align*}
        \sum_{\alpha \in G}
        \langle P_{\Lambda}U_{\alpha}\xi,AP_{\Lambda}U_{\alpha}\xi \rangle
        = \sum_{\alpha \in G} & \sum_{x \in B_{\Lambda}} \overline{\xi(x\alpha^{-1})} \left ( \sum_{y \in B_{\Lambda}} f(xy^{-1})\xi(y\alpha^{-1}) \right )\\
         = \sum_{\alpha \in G} &\sum_{x,y \in B_{\Lambda}} \overline{\xi(x\alpha^{-1})} f((x\alpha^{-1})(y\alpha^{-1})^{-1}) \xi(y\alpha^{-1})\\
         = \sum_{x',y'\in G} & \left (\overline{\xi(x')}f(x'(y')^{-1})\xi(y') \cdot \right . \\ & \left . \#\{(x,y,\alpha)\in B_{\Lambda}^2\times G;x\alpha^{-1}=x'\text{ and }y\alpha^{-1}=y'\} \right )\\ 
         = \sum_{x',y'\in G} &  \left (\overline{\xi(x')}f(x'(y')^{-1})\xi(y') \cdot \#(B_{\Lambda}(x')\cap B_{\Lambda}(y')) \right )\\
         = \#B_{\Lambda}\cdot & \langle \xi, r_{\Lambda}(A) \xi \rangle .
    \end{align*}
    Therefore $r_{\Lambda}$ is completely positive. 
    \qed
    }
\end{pf}

\begin{rmk}
\rm{
    Two unital and completely positive maps $q_{\Lambda}$ and $r_{\Lambda}$ commute with the derivatives in the sense that
    \begin{align*}
        q_{\Lambda}(df)=d_{\Lambda}(q_{\Lambda}f) \text{ and }r_{\Lambda}(d_{\Lambda}T)=d(r_{\lambda}T)
    \end{align*}
    for all $f \in \mathbb{C}[G]$ and $T \in P_{\Lambda}\mathbb{C}[G]P_{\Lambda}$. Especially these two maps are contractive with respect to Lip-norms $L_s$ and $L_{s,\Lambda}$.
    }
\end{rmk}

Note that 
\begin{align*}
    f-r_{\Lambda}\circ q_{\lambda}(f)&=f-\mathcal{F}_{\Lambda}(f)\\
    T-q_{\lambda}\circ r_{\Lambda}(T)&=(T_{xy^{-1}})_{x,y\in B_{\Lambda}}-(F_{\Lambda}(xy^{-1})T_{xy^{-1}})_{x,y\in B_{\Lambda}}.
\end{align*}
In Theorem 3.1, if $G$ has polynomial growth, we have already established the estimate
\begin{equation}
\label{C^*-estimate}
    \|f-r_{\Lambda}\circ q_{\Lambda}(f)\|_{C^*_r(G)}=\|f-\mathcal{F}_{\Lambda}(f)\|_{C^*_r(G)} \leq \epsilon_{\Lambda} L_s (f)
\end{equation}
with a positive integer $s>0$ and a sequence of positive numbers $\{\epsilon_{\Lambda}\}$ which converges to $0$. The other inequality is much more delicate because we do not have Haagerup inequality for truncation in the following sense. Even if $G$ satisfies the Haagerup inequality for constants $C$ and $s$ (i.e. $ \|\lambda(f)\|_{C^*_r(G)} \leq C\|f\|_{H^s} $ for all $f\in \mathbb{C}[G]$) we can not conclude 
\begin{align*}
     \|P_{\Lambda}\lambda(f)P_{\Lambda}\|_{B(P_{\Lambda }\ell^2(G))}\leq C\|P_{\Lambda}(1+D_{\Lambda})^sf\|_{P_{\Lambda}\ell^2(G)}
\end{align*}
in general. For example, if $f=\delta _x$ for $x \in G$ with $\Lambda <\ell(x) \leq 2\Lambda$ then the right hand side is $0$ but $ P_{\Lambda}\lambda(f)P_{\Lambda} $ is a nonzero operator. 
But actually, we can prove the following Lemma.

\begin{lem}
\rm{
    For any group $G$ with polynomial growth, there exists positive integer $s>0$ and a sequence $\{\epsilon_{\Lambda}\}_{\Lambda}$ of positive number which converges to $0$ as $\Lambda$ goes to infinity such that
    \begin{align*}
         \|T-q_{\lambda}\circ r_{\Lambda}(T)\|_{B(P_{\Lambda}\ell^2(G))} \leq \epsilon_{\Lambda}  L_{s,\Lambda}(T).
    \end{align*}
    }
\end{lem}

\begin{pf}
\rm{
    Since $G$ has polynomial growth, by Lemma 3.4, there exist constant $C$ and $\beta>0$ such that $\frac{\#(\partial B_{\Lambda})}{\# B_{\Lambda}}\leq \frac{C}{\Lambda^{\beta}}$. We can take a constant $d$ such that $(\Lambda^{\beta})^d$ dominates $\#B_{\Lambda}$. Since groups with polynomial growth has rapid decay property,
    there exists $s>0$ and $C'>0$ such that every $f\in \mathbb{C}[G]$ satisfies the Haagerup inequality;
    \begin{align*}
        \|\lambda(f)\|_{B(\ell^2(G))} \leq C'\sum_{x\in G} (1+\ell(x))^{2s}|f(x)|^2.
    \end{align*}
    Let $f\in \mathbb{C}[G]$ be such that $\text{supp}(f)\subset B_{2\Lambda}$. For $ x \in B_{\Lambda}$ we have
    \begin{align*}
        (d_{\Lambda}^{s+d+1}P_{\Lambda}\lambda(f)P_{\Lambda}) \delta_x =\sum_{y\in B_{\Lambda}} \ell(yx^{-1})^{s+d+1}f(yx^{-1})\delta_{y}.
    \end{align*}
    Therefore, by denoting $A:=P_{\Lambda}\lambda(f)P_{\Lambda}$ we have
    \begin{equation}
    \label{evaluation}
        L_{s+d+1,\Lambda}(A)^2= \|d_{\Lambda}^{s+d+1}(A)\|^2\geq \sum_{y \in B_{\Lambda}(x^{-1})} \ell(y)^{2s+2d+2}|f(y)|^2 .  
    \end{equation}
    By taking the sum of \eqref{evaluation} for all $x \in B_{\Lambda}$,  
    \begin{equation}
    \label{mainestimate}
    \begin{split}
        \#B_{\Lambda} \cdot L_{s+d+1,\Lambda}(A)^2\geq &  \sum_{x \in B_{\Lambda}} \sum_{y \in B_{\Lambda}(x^{-1})} \ell(y)^{2s+2d+2}|f(y)|^2\\
        =&\sum_{y \in B_{2\Lambda}} \#(B_{\Lambda}\cap B_{\Lambda}(y)) \cdot \ell(y)^{2s+2d+2}|f(y)|^2\\
        =& \sum_{\ell(y)\leq {\Lambda}^{\frac{1}{2}\beta}} \#(B_{\Lambda}\cap B_{\Lambda}(y)) \cdot \ell(y)^{2s+2d+2}|f(y)|^2\\ &+ \sum_{{\Lambda}^{\frac{1}{2}\beta}<\ell(y)\leq 2\Lambda} \#(B_{\Lambda}\cap B_{\Lambda}(y)) \cdot \ell(y)^{2s+2d+2}|f(y)|^2  \\
        \geq& \sum_{\ell(y)\leq {\Lambda}^{\frac{1}{2}\beta}} \#(B_{\Lambda})(1-C\frac{1}{\Lambda^{\beta}}\cdot {\Lambda}^{\frac{1}{2}\beta}) \ell(y)^{2s+2d+2}|f(y)|^2\\
        &+ \sum_{{\Lambda}^{\frac{1}{2}\beta}<\ell(y)\leq 2\Lambda} \ell(y)^{2s+2d+2}|f(y)|^2  \\
        \geq & \min\{\#(B_{\Lambda})(1-C\frac{1}{\Lambda^{\beta}}\cdot {\Lambda}^{\frac{1}{2}\beta}), (\Lambda^{\frac{1}{2}\beta})^{2d}\} \sum_{y\in B_{2\Lambda}} \ell(y)^{2s+2}|f(y)|^2\\
        \geq & \min\{\#(B_{\Lambda})(1-C\frac{1}{\Lambda^{\frac{1}{2}\beta}}), (\Lambda^{d\beta})\} \sum_{y\in B_{2\Lambda}} \ell(y)^{2s+2}|f(y)|^2.
        \end{split}
        \end{equation}
    For the second inequality, we use the formula \eqref{boundary} and the estimate for the volume of the boundary. By the choice of $d$, both term in $\min$ is larger than $\#B_{\Lambda}$ up to a constant depending only on $G$.
    Since we know from the estimate of \eqref{sobolev} that there exists a sequence ${\epsilon_{\Lambda}}$ converges to $0$ such that
    \begin{align*}
        \|q_{\Lambda}\lambda(f)&-q_{\Lambda}\circ r_{\Lambda}\circ q_{\Lambda}(\lambda(f))\|_{B(P_{\Lambda}\ell^2(G))}\\
        &\leq \|\lambda(f)-r_{\Lambda}\circ q_{\Lambda}(\lambda(f))\|_{B(\ell^2(G))}\leq \epsilon_{\Lambda} \sum_{y\in B_{2\Lambda}} \ell(y)^{2s+2}|f(y)|^2
    \end{align*}
    for any $f\in \mathbb{C}[G]$ supported in $B_{2\Lambda}$, by combining this inequality with \eqref{mainestimate}, we have the desired estimate.
    \qed
}
\end{pf}
 Applying Lemma 5.1, Remark 5.1, \eqref{C^*-estimate} and Lemma 5.2 to lemma 2.1, we have proven Theorem 5.1.

Finally, we study the natural question whether $(P_{\Lambda}\mathbb{C}[G]P_{\Lambda})$ converges to the (reduced) group $C^*$-algebra $C^*_r(G)$ or not under the same setting as Theorem 5.1. To answer it, we have to make it clear how to extend the Lip-norm $L_s$ on $\mathbb{C}[G]$ to its completion $C^*_r(G)$ following Chapter 4 of the paper of Rieffel \cite{rieffel1999metrics}.
Let $\mathcal{L}_1$ be the ball with respect to the Lip-norm $L_s$ i.e,
\begin{align*}
    \mathcal{L}_1:=\{f \in \mathbb{C}[G]:L_s(f)\leq 1\}.
\end{align*}
We denote the norm closure of $\mathcal{L}_1$ in $C^*_r(G)$ (rather than in $\mathbb{C}[G]$) by $\overline{\mathcal{L}}_1$. Then we define a (densely defined) seminorm $\overline{L}$ on $C^*_r(G)$ by
\begin{align*}
    \overline{L}_s(f)=\inf \{r\in \mathbb{R}_+:\text{ }f \in r \overline{\mathcal{L}}_1\}.
\end{align*}
Then it is shown in Proposition 4.4 of \cite{rieffel1999metrics} that if $L_s$ is a Lip-norm then $\overline{L}_s$ is a Lip-norm on $C^*_r(G)$ in the sense of Definition 2.1 which extends $L_s$. Using $\overline{L}_s$, we can show the $C^*$-algebra version of Theorem 5.1.

\begin{cor}
    \rm{
     If $G$ is a finitely generated group with polynomial growth, then there exists $s$ such that we have the quantum Gromov-Hausdorff convergence 
     \begin{align*}
     \text{dist}_q((C^*_r(G),\overline{L}_s),(P_{\Lambda}\mathbb{C}[G]P_{\Lambda},L_{s,\Lambda}))\rightarrow 0.
 \end{align*}
    }
\end{cor}

\begin{pf}
    \rm{
    We only need to show that
    \begin{align*}
        \|f-r_{\Lambda}\circ q_{\Lambda}(f)\| \leq \epsilon_{\Lambda} \overline{L}_s(f)
    \end{align*}
    for all $f\in C^*_r(G)$ and a sequence of positive numbers $\{\epsilon_{\Lambda}\}$ appeared in \eqref{C^*-estimate}. Let us assume $R:=\overline{L}_s(f)< \infty$. Then there exists a sequence $\{f_j\}$ in $R\mathcal{L}_1$ such that $\|f-f_j\|_{C^*_r(G)}\rightarrow 0$. Therefore
    \begin{align*}
        \|f-r_{\Lambda}\circ q_{\Lambda}(f)\|&=\lim_j \|f_j-r_{\Lambda}\circ q_{\Lambda}(f_j)\|\\
        & \leq \limsup_j \epsilon_{\Lambda} L_s(f_j) \leq \epsilon_{\Lambda} R=\epsilon_{\Lambda} \overline{L}_s(f).
    \end{align*}
    \qed
    }
\end{pf}

\section*{Acknowledgement}
The author would like to thank his advisor Professor Guoliang Yu for encouragements and discussions. He is also grateful to Jinmin Wang for discussing technical difficulties.

\bibliographystyle{plain}
\bibliography{main}

\end{document}